\documentclass[12pt]{amsart}
\usepackage{enumerate}
\usepackage{amsthm}
\usepackage{amsmath,amssymb,latexsym,amscd}
\usepackage{tikz}
\makeatletter
\@namedef{subjclassname@2020}{%
  \textup{2020} Mathematics Subject Classification}
\makeatother
\theoremstyle{definition}
\newtheorem{theorem}{Theorem}[section]

\newtheorem{proposition}[theorem]{Proposition}

\theoremstyle{definition}
\theoremstyle{definition}
\newtheorem{definition}[theorem]{Definition}
\newtheorem{remark}[theorem]{Remark}
\newtheorem{example}[theorem]{Example}
\usepackage{indentfirst}
\usepackage[colorlinks=true, urlcolor=Blue, pdfborder={0 0 0}]{hyperref}

\begin{document}
\baselineskip=17pt
\title[]{Fixed Point Sets of Involutions on The Product of Three Spheres }
\author[Dimpi and Hemant Kumar Singh]{ Dimpi and Hemant Kumar Singh}
\address{ Dimpi \newline 
	\indent Department of Mathematics\indent \newline\indent University of Delhi\newline\indent 
	Delhi -- 110007, India.}
\email{dimpipaul2@gmail.com}
\address{  Hemant Kumar Singh\newline\indent 
Department of Mathematics\newline\indent University of Delhi\newline\indent 
Delhi -- 110007, India.}
\email{hemantksingh@maths.du.ac.in}

\date{}
\thanks{}
\begin{abstract} 

\noindent Let $G=\mathbb{Z}_2$ act on a finitistic space $X$  having  mod $2$  cohomology of the product of three spheres $\mathbb{S}^n\times \mathbb{S}^m \times \mathbb{S}^l, 1 \leq n \leq m \leq l$. In this paper, we have determined the fixed point sets  of involutions on $X.$ This generalizes J. C. Su \cite{j1} results for involutions on the product of two sphere $\mathbb{S}^n \times \mathbb{S}^m, n\leq m.$
\end{abstract}
\thanks{ The first author of the paper is  supported by SRF of UGC, New Delhi, with  reference no.: 201610039267.}



\maketitle
\section {Introduction}
\indent Let $(G,X)$ be a transformation group, where $G$ is compact Lie group  and $X$ is compact Hausdroff space,  with the fixed point set $F.$ The study of the cohomological structure of the fixed point set has been an interesting problem in transformation groups. J. C. Su \cite{j1} gave a conjecture which  states  that if $G=\mathbb{Z}_p,p$ a prime, act on a finitistic space $X$ which satisfies poincar\'{e} duality with respect to  \v{C}ech cohomology  then each component of the fixed point set also satisfies poincar\'{e}  duality. Bredon \cite{Bredon} proved Su's conjecture, in the case when $X$ is totally nonhomologous to zero in $X_G$ (Borel space).  In a particular  case, Puppe \cite{puppe} has proved the Bredon's  conjecture: if $H^*(X,\mathbb{Z}_p)$ generated by  $k$ elements as an algebra, then any component $F_0$ of $F,$ $H^*(F_0,\mathbb{Z}_p)$ generated by  at most $k$ elements. Puppe     proved  that  if $X$ is totally nonhomologous to zero in $X_G,$ then the number of generators of the cohomology ring of each component of the fixed point set with $\mathbb{Z}_p$-coefficient  is  at most the number of generators of $H^*(X,\mathbb{Z}_p).$

 For the fixed point sets of actions of a group $G$ on the spheres $\mathbb{S}^n$ and on the product of two spheres $ \mathbb{S}^n\times \mathbb{S}^m,$ a lot of work has been done in the literature. Here, we recall some of the  results.  The fixed point sets of $G=\mathbb{Z}_p,$ $p$ a prime, or $G=\mathbb{T}^k,k\geq 1$ ($\mathbb{T}^k$ denotes $k$-dimensional torus) on a finitistic space $X$ having mod $p$ or integral cohomology $n$-sphere are mod $p$ or integral  cohomology $r$-sphere, where $-1\leq r \leq n$ and $n-r$ is even for $G=\mathbb{Z}_p,p>2$ and $G=\mathbb{T}^k$ (see Chapter III, \cite{bredon}).  According to this well known result, the fixed point sets of transformation group $(G, X_1\times X_2)$ (with diagonal action), where $X_1$ and $X_2$ has mod $p$ or integral cohomology  of spheres, has the same mod $p$ or integral cohomology algebra of the product of two spheres. Swan \cite{swan} gave  sufficient conditions which insure that the fixed point set of $G=\mathbb{Z}_p, ~p$ a prime, actions on mod $p$ cohomology product of two even dimensional spheres have  mod $p$ cohomology  of product of two even dimensional spheres of lower dimensions. Su \cite{j1} determined the fixed point sets of $G=\mathbb{Z}_p$ actions on a space $X$ with mod $p$ cohomology algebra product of spheres $\mathbb{S}^m \times \mathbb{S}^n.$ Further, Chang et al. \cite{c} and David \cite{devid} discussed the fixed point sets of circle actions and  torus actions on the product of  even dimensional spheres and the product of odd dimensional spheres, respectively. Allday \cite{allday} discussed the fixed point sets of torus actions  on cohomology product of three odd dimensional spheres. In continuation, it seems to be an interesting problem to determine the fixed point sets of involutions  on a finitistic space $X$ having mod $2$ cohomology of the arbitrary product of three spheres $\mathbb{S}^n \times  \mathbb{S}^m \times \mathbb{S}^l,$ $1\leq n \leq m \leq l.$
 
In this paper, we have determined  the possibilities of the fixed point sets of involutions  on a finitistic space $X$ having mod $2$ cohomology of the arbitrary product of three spheres $\mathbb{S}^n \times  \mathbb{S}^m \times \mathbb{S}^l,$ $1\leq n \leq m \leq l.$

\section{Preliminaries}  
	\noindent We  recall  some known facts that will be used in this paper. A paracompact Hausdroff space is called finitistic if every open covering has a finite dimension refinement. For example: (1) Compact Hausdroff spaces are finitistic spaces. (2) $\prod_{n\geq 1}(\mathbb{S}^n \times \mathbb{R}^k) $ is noncompact finitistic space.  (3) Paracompact spaces with finite covering dimension are finitistic spaces.  Let $G$ be a finite cyclic group  acting on a finitistic space $X$ with the fixed point set $F.$ Note that $G$ act freely on contractible space $E_G$ (an infinite join of $G$ with itself) and  the orbit map $ E_G \rightarrow B_G$ is the universal G-bundle. The associated fibration $ X \stackrel{i} \hookrightarrow X_G \stackrel{\pi} \rightarrow B_G$ is called the Borel fibration, where $X_G = (X\times E_G)/G$ (Borel space) and  $B_G=E_G/G$ (classifying space). If $F$ is nonempty and  $x \in F,$ then $\eta_x : B_G \hookrightarrow X_G$ is cross section of $\pi,$ where $B_G \approx \{x\} \times_{G} E_G ,$ and  we get  $H^*(X_G)= ker~ \eta_x^* \oplus im~ \pi^*.$  The induced homomorphism $\eta_x^*$ depends on the component $F_0$ of  $F$ in which $x$ lies. If $\alpha \in H^n(X_G)$ such that $\alpha \in Ker~ \eta_x^*,$ then the image of $\alpha$ under the  restriction of $j: F_G \hookrightarrow X_G$ on $(F_0)_G$ does not  involve the term of $H^n(B_G)\otimes H^0(F_0).$ For details, we refer \cite{bredon,Bredon}.\\ A space $X$ is said to be  totally nonhomologous to zero (TNHZ)
in $X_G$  with respect to the ring $R$ if the inclusion map $i:X \hookrightarrow X_G$ induces
a surjection in the cohomology $i^*: H^*(X_ G ; R) \rightarrow H^*(X;R). $ So, if $X$ is TNHZ in $X_G,$ then for $a\in H^n(X)$ there exist $\alpha'\in H^n(X_G)$ such that $i^*(\alpha')=a.$ Define $\alpha=\alpha'-\pi^*\eta_x^*(\alpha').$ Then it is easy to see that $i^*(\alpha)=a$  and $\eta_x^*(\alpha)=0.$ We have used the following propositions to determine the fixed point sets of product three of spheres.

\begin{proposition}(\cite{bredon})
	Let $G=\mathbb{Z}_2$  act on  finitistic space $X$ and $\sum$ rk $H^i(X,\mathbb{Z}_2)< \infty,$ then the following statements are equivalent: \\
\indent	(a) $X$ is TNHZ in $X_G.$   \\
\indent	(b) $\sum_{i \geq 0}$ rk $H^i(F,\mathbb{Z}_2) =\sum_{i \geq 0}$ rk $H^i(X,\mathbb{Z}_2).$ \\ 
	\indent (c) $\pi_1(B_G)$ acts trivially on $H^*(X; \mathbb{Z}_2)$ and spectral sequence $E^{r,q}_2$ of fibration $X \hookrightarrow X_G \rightarrow B_G$ degenerates. 
\end{proposition} 
\begin{proposition}\label{thm 2}(\cite{bredon})
	 Let $G = \mathbb{Z}_2$  act on the finitistic
	space $X.$ Suppose that $A$ be a closed and invariant subset of $X,$ and  $H^i ( X ,A; \mathbb{Z}_2)
	= 0$ for $i > n$. Then
	$j ^* : H^ k ( X_G,A_G ; \mathbb{Z}_2) \rightarrow H^k(F_G, A_G\cap F_G; \mathbb{Z}_2)$
	is an isomorphism for $k > n$. If $(X,A) $ is TNHZ
	 in $(X_ G,A_G) $, then $j^*$ is a monomorphism for all $k.$
\end{proposition}
	\begin{proposition}(\cite{Bredon1})\label{2.2}
	Let $X$ be TNHZ in $X_G$ and $\{\gamma_j\}$ be a set of homogeneous elements in $H^*(X_G;\mathbb{Z}_2)$ such that $\{i^*(\gamma_j)\}$ forms $\mathbb{Z}_2$-basis of $H^*(X; \mathbb{Z}_2).$ Then, $H^*(X_G;\mathbb{Z}_2)$ is the free $H^*(B_G)$-module generated by $\{\gamma_j\}.$
\end{proposition}
\begin{proposition}(\cite{bredon})\label{prop1}
	For a finitistic $G$-space $X,$ where $G$ is cyclic group of prime order $p,$  we have   $ \sum_{i \geq j}$ rk $H^{i}(F;\mathbb{Z}_2) \leq \sum_{i \geq j} $ rk $H^{i}(X;\mathbb{Z}_2),$ for each $j.$
\end{proposition}

\begin{definition}{\cite{bredon}}
A space $X$ is said to be poincar\'{e} duality space over a field $k$ of formal dimension $r$ if the following conditions are satisfied:\\
\indent(a) $H^*(X;k)$ is finitely generated.\\ \indent(b) $H^i(X;k)=0$ for $i>r$ and $H^r(X;k)\approx k.$\\ \indent(c) An element $u\in H^i(X;k)$ is nonzero if and only if there exists an element $u'\in  H^{r-i}(X;k)$ called poincar\'{e} dual of $u,$ such that the cup product $ 0 \neq uu'\in H^r(X;k).$	
\end{definition}
Recall that the poincar\'{e} duals  of basis elements are unique \cite{char}.
\begin{proposition} \label{thm 1}(\cite{bredon})
	Let $X$ be a finitistic poincar\'{e} duality space over $\mathbb{Z}_2$ of formal dimension $n$, where $p$ is prime. Let $G=\mathbb{Z}_2$ act on $X,$ where $X$ is TNHZ in $X_G.$ Then, for each component $F_0$ of $F$ is a poincar\'{e} duality space of formal dimension $r \leq n.$  If $r=n,$ then $F=F_0$ is connected  and  the restriction $H^*(X;\mathbb{Z}_2)\rightarrow H^*(F;\mathbb{Z}_2)$ are isomorphisms.
\end{proposition}
\begin{proposition}\cite{puppe}\label{puppe}
	Let $G=\mathbb{Z}_2$ act on a finitistic space $X,$ where $X$ is TNHZ in $X_G.$ If $H^*(X ;\mathbb{Z}_2)$ is generated by $k$ elements as an algebra, then each component of the fixed point set is generated by at most $k$ elements.
\end{proposition}
\begin{proposition}\label{nontrivial}(\cite{bredon})
	Let $G=\mathbb{Z}_2$ act on a finitistic space $X$ and  $\pi_1(B _G)$ acts nontrivially on $H^*(X).$  Then,  the $E_2$ term of the Leray-Serre spectral sequence of the fibration  $ X \stackrel{i} \hookrightarrow X_G \stackrel{\pi} \rightarrow B_G$ is given by
	\[
	E_2^{k,i}=
	\begin{cases}
		\text{ker $\tau$} & \mbox{for}~ ~k=0, \\
		\text{ker $\tau$/im $\sigma$} & \text{for } k>0,\\

	\end{cases}
	\]
	where $\tau = \sigma = 1+g^*,$ $g^*$ is induced by a generator $g$ of $G.$
\end{proposition}
	\begin{proposition}\label{prop 6}(\cite{bredon})
	Let $G=\mathbb{Z}_2$ act on a finitistic space $X.$ Then, for any $a \in H^n(X),$ the element  $ag^*(a)$ is permanent cocycle in spectral sequence of $X_G \rightarrow B_G.$
\end{proposition}

	\noindent Recall  that 
$H^*(\mathbb{S}^n \times \mathbb{S}^m \times \mathbb{S}^l; \mathbb{Z}_2)=\mathbb{Z}_2[a,b,c]/<{a^{2},b^2, c^2}>,$ where deg $a=n,$ deg $b=m$ and  deg $c=l, ~1 \leq n \leq m\leq l.$\\
\noindent Throughout the paper,  
$H^*(X)$ will denote the \v{C}ech cohomology of a space $X$ with coefficient group $G=\mathbb{Z}_2,$ and   $X\sim_2 Y,$  means $H^*(X;\mathbb{Z}_2 )\cong H^*(Y;\mathbb{Z}_2).$

\section{Fixed point Sets of involutions on the product of three spheres}
Let $G=\mathbb{Z}_2$ act on a finitistic space  $X\sim_2 \mathbb{S}^n \times \mathbb{S}^m \times \mathbb{S}^l,$ $1\leq n \leq m \leq l.$ If the fixed point set is empty, then the cohomology ring of the orbit space $X/G$ has been discussed in \cite{dimpi}.
In this section, we determine the possibilities of nonempty fixed points sets of involutions on $X.$ First, we determine the possibilities of nonempty connected fixed point sets of involutions on $X.$
\subsection{Connected fixed point sets of involutions on $X$}
\noindent Suppose that  $X$ is TNHZ in $X_{\mathbb{Z}_2},$ then we have 
  $\sum$ rk $H^i(F) =\sum$ rk $H^i(X)=8.$ By Proposition \ref{thm 1}, we get  $F$  is   poincar\'{e} duality space  of formal dimension $r\leq n+m+l.$  Let $0 < q_1 \leq q_2 \leq q_3\leq q_4\leq q_5\leq q_6 < r$ be the  non vanishing dimensions of   $  H^*(F),$ and ${u_i}$ is generator of $q_i^{th}$ cohomology group of $F,1 \leq i \leq 6$ and $v$ is generator of $H^r(F).$ Assume  that  $\alpha\in H^n(X_G), \beta\in H^m(X_G)$ and $\gamma\in H^l(X_G)$ represents generators  $a,b$ and $c$ of $H^*(X,x),$ respectively, such that  $\eta_x(\alpha)=\eta_x(\beta)=\eta_x(\gamma)=0.$ By Proposition \ref{2.2},  we get $\{1,\alpha, \beta, \gamma, \alpha\beta ,\alpha\gamma,\beta\gamma,\alpha\beta\gamma \}$ forms a basis for $H^*(X_G)$ over $H^*(B_G)$-module. Note that the images of generators of $H^*(X_G)$  under the homomorphism $j^*:H^*(X_G)\rightarrow H^*(F_G)$ involves generators of $H^*(F).$ Since $a^2=b^2=c^2=0,$ we get $\alpha^2,\beta^2$ and $\gamma^2$ may be zero or nonzero, and hence $u_iu_j,1\leq i,j\leq 6,$ may be zero or nonzero. Further, if each generator $u_i$ is poincar\'{e} dual of itself or any two generators are poincar\'{e} dual of itself or any four generators are poincar\'{e} dual of itself, then we get $H^*(F)$ must have  more than three  generators, which contradicts Proposition \ref{puppe}.  As $q_i\leq q_{i+1},1\leq i\leq 5,$ we get $u_1, u_2$ and $ u_3$ are  poincar\'{e} duals of $u_6, u_5$ and $ u_4,$   respectively. So, we have $u_1u_6=u_2u_5=u_3u_4 =v \in H^r(F).
$ Clearly, $q_4=r-q_3,q_5=r-q_2$ and $q_6=r-q_1,$  and $u_2u_6=u_3u_5=u_3u_6=u_4u_5=u_4u_6=u_5u_6=v^2=0~\&~u_iv=0, 1\leq i \leq 6.$  Thus, the possibilities of the fixed point set will depend on the cup products   $u_1u_2,u_1u_3,u_1u_4,u_1u_5,u_2u_3$ and $u_2u_4$ which may be zero or nonzero. So, we have considered the cases depending upon how many cup products out of these six are simultaneously nonzero.
\begin{theorem}\label{3.1}
	Let $G=\mathbb{Z}_2$ act on $X\sim_2 \mathbb{S}^n \times \mathbb{S}^m \times \mathbb{S}^l,$ where $1\leq n \leq m \leq l$ and $X$ is TNHZ in $X_G.$ If $F$  is connected and none of the  cup products $u_1u_2,u_1u_3,u_1u_4,u_1u_5,u_2u_3$ and $u_2u_4$ of generators of $H^*(F)$ is nonzero, then $F \sim_2\underset{3}{\#} \mathbb{P}
	^3(q), q\leq n$ and $q=1,2,4,8.$
	\end{theorem}
\begin{proof}
As discussed above, we have $H^*(X_G)$ is the free  $H^*(B_G)$-module with basis $\{1,\alpha, \beta, \gamma, \alpha\beta ,\alpha\gamma,\beta\gamma,\alpha\beta\gamma \}.$ Then we write	$$j^*(\alpha) = A_1t^{n-q_1}\otimes u_1 + A_2t^{n-q_2}\otimes {u_2}+ A_3t^{n-q_3}\otimes {u_3}+ A_4t^{n-q_4}\otimes {u_4}+ $$ $$A_5t^{n-q_5}\otimes  {u_5}+ A_6t^{n-q_6}\otimes {u_6}+ A_7t^{n-r}\otimes {v},$$ 
	$$j^*(\beta) = A_1^{'}t^{m-q_1}\otimes u_1 + A_2^{'}t^{m-q_2}\otimes {u_2}+ A_3^{'}t^{m-q_3}\otimes {u_3}+ A_4^{'}t^{m-q_4}\otimes {u_4}+ $$ $$A_5^{'}t^{m-q_5}\otimes  {u_5}+ A_6^{'}t^{m-q_6}\otimes {u_6}+ A_7^{'}t^{m-r}\otimes {v} ,$$ 
	$$j^*(\gamma) = A_1^{''}t^{l-q_1}\otimes u_1 + A_2^{''}t^{l-q_2}\otimes {u_2}+ A_3^{''}t^{l-q_3}\otimes {u_3}+ A_4^{''}t^{l-q_4}\otimes {u_4}+ $$ $$A_5^{''}t^{l-q_5}\otimes  {u_5}+ A_6^{''}t^{l-q_6}\otimes {u_6}+ A_7^{''}t^{l-r}\otimes {v} ,$$ where all $A_i,A'_i,A''_i \in \mathbb{Z}_2~\forall~ 1\leq i\leq 7.$ As all cup products are zero, we get $u_4^2=u_5^2=u_6^2=0.$ So, we have  
	$$j^*(\alpha \beta \gamma) = A_1 A_1^{'}A_1^{''}t^{l-q_1}\otimes u_1^3 +  A_2 A_2^{'}A_2^{''}t^{l-q_2}\otimes {u_2^3}+  A_3 A_3^{'}A_3^{''}t^{l-q_3} \otimes u_3^3.$$
	As $j^*$ is injective and $\alpha\beta\gamma \neq 0,$ we get  $j^*(\alpha \beta \gamma) \neq 0.$ Thus, at least one of  $u_1^3,u_2^3$ and $u_3^3$ must be nonzero. Now, we consider following cases: \\ 
\textbf{Case (i)}~ Any one of $u_1^3,u_2^3$ and $u_3^3$ is nonzero.\\
	Assume $u_1^3 \neq 0$ and $u_2^3 =u_3^3 = 0.$ We must have $u_2^2=u_3^2=0$ and  $u_6=u_1^2.$ Thus, we get  $u_1^3=u_3u_4=u_2u_5=v$ is generator of  $H^r(F).$ Clearly, the cohomology algebra  of the  fixed point set $F$ has five generators which contradicts Proposition \ref{puppe}. 
So, this case not possible.
	\\
\textbf	{Case (ii)}~ Any two of $u_1^3,u_2^3$ and $u_3^3$ are nonzero.\\
	Assume $u_1^3 \neq 0,u_2^3 \neq 0$ and $u_3^3 = 0.$ Then, we must have $u_6=u_1^2, u_5 = u_2^2,$ $u_3^2=0$  and   $u_1^3=u_2^3=u_3u_4=v$ generator of  $H^r(F).$ Clearly, $q_1=q_2$ and $q_5=q_6.$ It is easy to observe that  the cohomological algebra of the fixed point set $F$ has four generators, a contradiction. So, this  case is also not possible. \\
\textbf{Case(iii)}~  $u_1^3,u_2^3$ and $u_3^3$ are nonzero.\\
	Clearly, $u_1^2,u_2^2$ and $u_3^2$ are  all  nonzero, and we must have   $u_4=u_3^2, u_5=u_2^2$ and $u_6=u_1^2.$ Therefore, $u_1^3=u_2^3=u_3^3 $ is generator of $H^r(F).$ In this case, we must have  $q_1=q_2=q_3.$ So, we get $u_1,u_2,u_3$ are generator of degree $q,$ and $u_1^2,u_2^2,u_3^2$ are generators of degree $2q.$ Thus, the cohomological algebra of fixed point set $F$ is given by $${\mathbb{Z}_2[u_1,u_2,u_3]}/{<u_1^4,u_2^4,u_3^4,u_1^3+u_2^3,u_2^3+u_3^3,u_1^3+u_3^3,u_1u_2,u_2u_3,u_1u_3>},$$ where deg $u_1=$ deg $u_2=$ deg $u_3=q.$  
	Thus, $F \sim_2 \underset{3}{\#} \mathbb{P}^3(q).$ By Proposition \ref{prop1}, we get $q\leq n,$ where   $q=1,2,3,4$ \cite{adam}. 
	\end{proof}
\begin{theorem}\label{3.2}
	Let $G=\mathbb{Z}_2$ act on $X\sim_2 \mathbb{S}^n \times \mathbb{S}^m \times \mathbb{S}^l,$ where $1\leq n \leq m \leq l$ and $X$ is TNHZ in $X_G.$ If $F$ is connected and one cup product from $u_1u_2,u_1u_3,u_1u_4,u_1u_5,u_2u_3$ and $u_2u_4$ of generators of $H^*(F)$ is nonzero, then $F \sim_2 \mathbb{P}^3(r_1) \# (\mathbb{P}^2(r_2)\times \mathbb{S}^{r_3} ),$ where $min~\{r_1,r_2\}\leq n $ and $r_1,r_2=1,2,4,8.$
\end{theorem}
\begin{proof}
First, assume that $u_1u_2 \neq 0$ and $u_1u_3=u_1u_4=u_1u_5=u_2u_3=u_2u_4=0.$\\ It is easy to see that
	either  $u_1u_2=u_5$ or $u_1u_2=u_6.$ Suppose that $u_1u_2=u_5.$ Then, we have  $u_1u_6=u_2^2u_1=u_3u_4$ is generator of $H^r(F).$ As all the cup products $u_1u_4,u_1u_5,u_2u_3$ and $u_2u_4$ are zero, we must have  $u_2^2=u_6$ and $u_4^2=0.$ Now, we consider four cases: (i) $u_1^2=u_3^2=0$ (ii) $u_1^2\neq 0, u_3^2=0$ (iii) $u_1^2=0,u_3^2\neq 0,$ and (iv) $u_1^2\neq 0,u_3^2\neq 0.$\\ \textbf{Case (i):} $u_1^2=u_3^2=0.$\\ In this case, the cohomology algebra of the fixed point set has four generators, namely, $u_1,u_2,u_3$ and $u_4,$ which contradicts Proposition \ref{puppe}. \\ \textbf{Case (ii):}  $u_1^2\neq 0$ and $u_3^2=0.$\\ Since, $u_1u_3=u_1u_4=u_1u_5=0,$ we must have  $u_6=u_1^2.$ So, $u_2^2u_1=u_1^3=u_3u_4$ is generator of $H^r(F).$ Thus, the cohomological algebra of the fixed point set is generated by four generators $u_1,u_2,u_3$ and $u_4,$ a contradiction. \\ \textbf{Case (iii):} $u_1^2= 0$ and $u_3^2\neq0.$\\ In this case, we get $u_4=u_3^2$ and $q_1=q_2=q_3.$ Thus, $H^{2q_1}(F)$ is generated by $u_3^2,u_1u_2,u_2^2,$ and $u_2^2u_1=u_3^3$ generates $H^{3q_1}(F).$  Hence, the cohomological  algebra of the fixed point set is given by  
	$$ {\mathbb{Z}_2[u_1,u_2,u_3]}/{<u_1^2,u_2^3,u_3^4,u_3^3+u_2^2u_1,u_1u_3,u_2u_3 >},$$ where deg $u_1=$ deg $u_2=$ deg $u_3=q_1.$ Thus, $F\sim_2 \mathbb{P}^3(q_1) \# (\mathbb{P}^2(q_1)\times \mathbb{S}^{q_1}).$ By Proposition \ref{prop1}, we get $ q_1\leq n$ and $q_1=1,2,4,8$ \cite{adam}.	
	\\ \textbf{Case (iv):}  $u_1^2\neq 0$ and $u_3^2\neq0.$\\ In this case,  we get $u_4=u_3^2$ and $u_6=u_1^2=u_2^2.$  So, $u_2^2u_1=u_1^3=u_3^3$ is generator of $H^r(F),$ and   $q_1=q_2=q_3.$ By the change of basis, $d=u_1+u_2,$ we get $d^2=0,$ $u_2d=u_2u_1+u_2^2\neq 0, u_6d=u_1u_6\neq 0.$ Clearly, $u_3d=u_4d=u_5d=0$ and $u_2^2d=u_3^2$ is generator of $H^{3q_1}(F).$ Hence,  $F\sim_2 \mathbb{P}^3(q_1) \# (\mathbb{P}^2(q_1)\times \mathbb{S}^{q_1} ).$  \\ Now, we  suppose   $u_1u_2=u_6.$ Then,  $u_1^2u_2=u_2u_5=u_3u_4$  generates   $H^r(F).$ Clearly, $u_5=u_1^2$ and   $u_4^2=0.$ Now, $u_2^2$ and $u_3^2$ are either zero or nonzero. Similarly, as discussed  above,  if  $\{u_2^2=0$ and $u_3^2=0\}$ or $\{u_2^2\neq 0$ and $u_3^2= 0 \},$ then the cohomology ring $H^*(F)$ has four generators, which is not possible.  Now, if $u_2^2=0$ and $u_3^2\neq 0,$ then  we  must have $u_4=u_3^2.$ Thus, $q_1=q_2=q_3,$ and $u_1^2u_2=u_3^3$  generates $H^{3q_1}(F).$ Hence, the cohomological algebra of the fixed point set is given by  
		$$ {\mathbb{Z}_2[u_1,u_2,u_3]}/{<u_1^3,u_2^2,u_3^4,u_3^3+u_1^2u_2,u_1u_3,u_2u_3 >},$$ where deg $u_1=$ deg $u_2=$ deg $u_3=q_1.$ Thus, $F\sim_2 \mathbb{P}^3(q_1) \# (\mathbb{P}^2(q_1)\times \mathbb{S}^{q_1} ),$  $ q_1\leq n.$	Finally, if $u_2^2$ and $u_3^2$ both are nonzero, then we get  $u_4=u_3^2$ and $u_5=u_1^2=u_2^2.$ By the change of basis, $d=u_1+u_2,$ we get $F\sim_2 \mathbb{P}^3(q_1) \# (\mathbb{P}^2(q_1)\times \mathbb{S}^{q_1} ),q_1\leq n.$
		
		Similarly, if $u_1u_3\neq 0$ or $u_2u_3\neq 0,$ then $F\sim_2  \mathbb{P}^3(q_1) \# (\mathbb{P}^2(q_1)\times \mathbb{S}^{q_1} ).$\\
		Next, assume that $u_1u_4\neq 0.$  Here, we must have  $u_1u_4=u_6,$ and $u_1^2u_4=u_2u_5=u_3u_4$ is generator of $H^r(F).$ Clearly, $u_4^2=u_5^2=0$ and $u_3=u_1^2.$ If $u_2^2=0,$ then cohomology algebra of the fixed point set has four generators, which is not possible. If  $u_2^2\neq 0,$ then $u_5=u_2^2$ and we get $F\sim_2 \mathbb{P}^3(q_2) \# (\mathbb{P}^2(q_1)\times \mathbb{S}^{q_4} ),$ $q_1\leq n $ and $q_1,q_2=1,2,4,8.$
		
		  Similarly, if $u_1u_5\neq 0,$ then $F\sim_2 \mathbb{P}^3(q_2) \# (\mathbb{P}^2(q_1)\times \mathbb{S}^{q_5} ),$  and if $u_2u_4\neq 0,$ then $F\sim_2 \mathbb{P}^3(q_1) \# (\mathbb{P}^2(q_2)\times \mathbb{S}^{q_4} ).$ This completes the proof.
	 \end{proof}
\begin{theorem} \label{3.3}
	Let $G=\mathbb{Z}_2$ act on $X\sim_2 \mathbb{S}^n \times \mathbb{S}^m \times \mathbb{S}^l,$ where $1\leq n \leq m \leq l$ and $X$ is TNHZ in $X_G.$ If $F$ is connected and two cup products from $u_1u_2,u_1u_3,u_1u_4,u_1u_5,u_2u_3$ and $u_2u_4$  of generators of $H^*(F)$ are nonzero, then $F$ must be  one of the following:
	\begin{enumerate}
		
		\item $F \sim_2 (\mathbb{P}^5(r_1) \#  ({\mathbb{S}^{r_2} \times \mathbb{S}^{r_3} }),$ where $r_1=1,2,4,8$ and $ min \{r_1,r_2,r_3\} \leq n.$
		\item $F \sim_2 (\mathbb{P}^2(r_1) \# \mathbb{P}^2(r_1))\times \mathbb{S}^{r_2},$ where $r_1=1,2,4,8$ and $min \{r_1,r_2\} \leq n. $
			\item $F \sim_2 {\mathbb{Z}_2[x,y,z]}/{<x^3,u^3,v^2,xv+u^2,yz >},$ where $(u,v)\in\{(y,z),(z,y)\}$ and $(\mbox{deg~} x,\mbox{deg~} y,\mbox{deg~} z)\in\{(q_1,q_2,q_3),(q_1,q_2,q_4),(q_2,q_1,q_3),(q_1,q_3,q_5),(q_3,q_1,q_2)\}.$

		\item $F \sim_2 {\mathbb{Z}_2[x,y,z]}/{<x^2,y^3,z^3,y^2+z^2+ux,yz >},$ where $u\in \{y,z\}$ and  $(\mbox{deg~} x,\\ \mbox{deg~} y,\mbox{deg~} z)\in\{(q_1,q_2,q_3),(q_2,q_1,q_3),(q_3,q_1,q_2)\}.$    
		
	\end{enumerate}
	
\end{theorem}
\begin{proof}
	Clearly, we have  $6 \choose 2$$=15$ cases depending on pairs of nonzero  cup products.	It is easy to see that out of these fifteen cases, the following six nonzero pairs are not possible: (1) $u_1u_2, u_2u_4$ (2) $u_1u_3, u_2u_4$ (3) $u_1u_4, u_1u_5$ (4) $u_1u_4, u_2u_3$ (5) $u_1u_5, u_2u_3,$ and (6) $u_1u_5, u_2u_4.$ 	A contradiction of these cases occurs from  the generators whose cup products are zero. Now, we consider remaining nonzero pairs of cup products:\\
	{Case (7):}  $u_1u_2\neq 0~\&~ u_1u_5\neq 0.$ \\
	As $u_1u_3=u_1u_4=u_2u_3=u_2u_4=0,$  we must have either $u_1u_5=u_1u_5=u_6$
	or \{$u_1u_2=u_5$ and  $u_1u_5=u_6$\}. First, suppose that $u_1u_2=u_1u_5=u_6.$ Then, $u_1^2u_2=u_1^2u_5=u_2u_5=u_3u_4$ is generator of $H^r(F).$ So, $u_1^2\neq 0.$ As $u_1u_3=u_1u_4=0, $ we have either $u_1^2=u_2$ or $u_1^2=u_5.$ Further, in either cases, $u_2$ or $u_5$ have  two poincar\'{e} duals, a contradiction. \\ Next, suppose that  $u_1u_2=u_5$ and $u_1u_5=u_6.$ Then, we have $u_1u_5=u_1^2u_2$  and $u_1^3u_2=u_2^2u_1=u_3u_4$ is generator of $H^r(F).$ Since, $u_1u_3=u_1u_4=0,$ we have $u_2=u_1^2.$ Thus, $u_1u_2=u_1^3,u_1^2u_2=u_1^4$ and $u_1^5=u_3u_4$ are generators of $H^{q_5}(F),H^{q_6}(F)$ and $H^r(F),$ respectively. Clearly, $u_4^2=0.$ If $u_3^2\neq 0,$ then we must have $u_4=u_3^2,$ and hence   $F\sim_2 \mathbb{P}^5(q_1)\#\mathbb{P}^3(q_3),$ where $5q_1=3q_3.$ By Adam \cite{adam}, this is not possible.  If $u_3^2= 0,$ then we get $F\sim_2 \mathbb{P}^5(q_1)\# (\mathbb{S}^{q_3}\times \mathbb{S}^{q_4}).$ By Proposition \ref{prop1}, $q_1\leq n.$ This realizes possibility (1).

	Similarly, for Case (8)  $u_1u_3\neq 0~\&~ u_1u_4\neq 0,$ we get $F\sim_2 \mathbb{P}^5(q_1)\# (\mathbb{S}^{q_2}\times \mathbb{S}^{q_5}),$ and  for Case (9) $u_2u_3\neq 0~\&~ u_2u_4\neq 0,$ we get $F\sim_2 \mathbb{P}^5(q_6)\# (\mathbb{S}^{q_1}\times \mathbb{S}^{q_6}).$ So, these cases also realizes possibility (1). 	 \\
	{Case (10):} $u_1u_4\neq 0~\&~ u_2u_4\neq 0.$\\
	In this case, we must have $u_2u_4=u_5$ and $u_1u_4=u_6.$ Thus, we have $u_1^2u_4=u_2^2u_4=u_3u_4$ is a generator of  $H^r(F).$ As $u_1u_2=u_1u_3=u_1u_5=u_2u_3=0,$ we get $u_1^2=u_2^2=u_3.$ Clearly, $q_1=q_2$ and $u_4^2=0.$ Hence, $F\sim_2 (\mathbb{P}^2(q_1)\# \mathbb{P}^2(q_1))\times \mathbb{S}^{q_4}, q_1\leq n.$ This realizes possibility (2).\\
	{Case (11):}	 $u_1u_2\neq 0$ and $u_1u_4\neq 0.$\\
	In this case, we get either $u_1u_4=u_1u_5=u_6$ or \{$u_1u_2=u_5$ and $u_1u_4=u_6$\}. If $u_1u_2=u_1u_4=u_6,$ then by uniquesness of poincar\'{e} dual, this is not possible.\\ Next, Suppose that $u_1u_2=u_5$ and $u_1u_4=u_6.$ Then, we have $u_1^2u_4=u_2^2u_1=u_3u_4$ generates of $H^{r}(F).$ Clearly,  $u_1^2=u_3,$  $u_2^2=u_6$ and $u_4^2=0.$ Consequently, $u_1^3=u_2^3=u_4^2=0.$ Hence, the cohomology ring of the fixed point set is isomorphic to $${\mathbb{Z}_2[u_1,u_2,u_4]}/{<u_1^3,u_2^3,u_4^2,u_1u_4+u_2^2,u_2u_4>} 
	$$ where deg $u_1=q_1,$  deg $u_2=q_2$ and  deg $u_4=q_4.$ This realizes possibility (3) for $(u,v)=(y,z),$ $x=u_1,y=u_2$ and $z=u_4.$
	
			Similarly, for Case (12)  $u_1u_3\neq 0~\&~ u_1u_5\neq 0,$  the cohomology ring of the fixed point set is isomorphic to $${\mathbb{Z}_2[u_1,u_3,u_5]}/{<u_1^3,u_3^3,u_5^2,u_1u_5+u_3^2,u_3u_5>} 
			$$ where deg $u_1=q_1,$  deg $u_3=q_3$ and  deg $u_5=q_5.$ This also realizes possibility (3) for $(u,v)=(y,z),$ $x=u_1,y=u_3$ and $z=u_5.$	 \\
{Case (13):} $u_1u_2\neq 0$ and $u_1u_3\neq 0.$\\
	In this case, we have following subcases (i) $u_1u_2=u_1u_3=u_6$ (ii) $u_1u_2=u_5~\&~u_1u_3=u_4$ (iii) $u_1u_2=u_5~\&~u_1u_3=u_6,$ and  (iv) $u_1u_2=u_6~\&~u_1u_3=u_4.$ Note that  subcase (i) is not possible.\\ Subcase (ii):  $u_1u_2=u_5$ and $u_1u_3=u_4.$ Then, we have $u_2^2u_1=u_3^2u_1=u_1u_6$  generates $H^r(F).$ It is easy to see that $u_6=u_3^2=u_2^2$ is generator of $H^{q_6}(F).$ Then, $q_2=q_3.$ Now, if $u_1^2=0,$ then $F\sim_2 (\mathbb{P}^3(q_2)\# \mathbb{P}^3(q_2))\times \mathbb{S}^{q_1}.$ If $u_1^2\neq 0,$ then we must have $u_1^2=u_6.$ So, we get $q_1=q_2=q_3.$ By the change of basis $d=u_1+u_2,$ we get  $F\sim_2 (\mathbb{P}^3(q_1)\# \mathbb{P}^3(q_1))\times \mathbb{S}^{q_1}.$ These realizes possibility (2).\\
	Subcase (iii): $u_1u_2=u_5$ and $u_1u_3=u_6.$ Then,  $u_1^2u_3=u_2^2u_1=u_3u_4$  generates $H^r(F).$ Clearly, $u_4=u_1^2$ and $u_1u_3=u_2^2.$ Consequently,  $u_1^3=u_2^3=0.$ If $u_3^2=0,$ then the cohomology ring of the fixed point set  is isomorphic to 
	$${\mathbb{Z}_2[u_1,u_2,u_3]}/{<u_1^3,u_2^3,u_3^2,u_1u_3+u_2^2,u_2u_3 >}, 
	$$ where deg $u_1=q_1,$  deg $u_2=q_2$ and  deg $u_3=q_3.$ This  realizes possibility (3) for $(u,v)=(y,z),$  $x=u_1,y=u_2$ and $z=u_3.$ If $u_3^2\neq 0,$ then we get $u_4=u_3^2=u_1^2.$ Thus, $q_1=q_2=q_3,$ and by change of basis $d=u_1+u_3,$ we get $\{d,u_2,u_3\}$ generates $H^{q_1}(F),$ $\{u_3d,u_2d,u_2^2=u_3^2+u_3d\} $ generates $H^{2q_1}(F)$ and $u_2^2d=u_3^2d$ generates $H^{3q_1}(F).$  Thus,  the cohomology ring $H^*(F)$ is isomorphic to $${\mathbb{Z}_2[d,u_2,u_3]}/{<d^2,u_2^3,u_3^3,u_3d+u_2^2+u_3^2,u_2u_3 >}, 
	$$ where deg $d=$deg $u_2=$deg $u_3=q_1.$ This realizes possibility (4) for $u=z,$  $x=d,y=u_2$ and $z=u_3.$\\
	 Subcase (iv):  $u_1u_2=u_6$ and $u_1u_3=u_4.$ Then,   $u_1^2u_2=u_3^2u_1$ generates $H^r(F).$ It is easy to see that  $u_5=u_1^2$ and $u_6=u_3^2.$ If $u_2^2=0,$ then the cohomology ring of the fixed point set is isomorphic to $${\mathbb{Z}_2[u_1,u_2,u_3]}/{<u_1^3,u_2^2,u_3^3,u_1u_2+u_3^2,u_2u_3 >},$$ where deg $u_1=q_1,$  deg $u_2=q_2$ and  deg $u_3=q_3.$ This  realizes possibility (3) for $(u,v)=(z,y)$ and $x=u_1,y=u_2~\&~z=u_3.$ Now, if $u_2^2\neq 0,$ then $u_5=u_1^2=u_2^2,$ and again by the  change of basis $d=u_1+u_2,$ we get $H^*(F)$ is isomorphic to $${\mathbb{Z}_2[d,u_2,u_3]}/{<d^2,u_2^3,u_3^3,u_2d+u_2^2+u_3^2,u_2u_3 >}, 
	 $$ where deg $d=$ deg $u_2=$ deg $u_3=q_1.$ This realizes possibility (4) for $u=y,$ $x=d,y=u_2$ and $z=u_3.$

	Similarly, for Case (14)  $u_1u_3\neq 0~\&~ u_2u_3\neq 0,$ we realize possibility (2), possibility (3) for $(u,v)=(y,z) ~\&~(z,y);$  $x=u_2,y=u_1~\&~z=u_3,$ and  possibility (3) for $u=y~\&~z;$ $x=u_2,y=d~\&~z=u_3,$  and for Case (15)  $u_1u_2\neq 0~\&~ u_2u_3\neq 0,$  we realize possibility (2), possibility (3) for $(u,v)=(y,z)~\&~(z,y);$  $x=u_3,y=u_2~\&~z=u_1,$  and possibility (3) for $u=y~\&~z;$ $x=u_1,y=u_2~\&~z=d.$
\end{proof}
\begin{theorem}\label{3.4}
	Let $G=\mathbb{Z}_2$ act on $X\sim_2 \mathbb{S}^n \times \mathbb{S}^m \times \mathbb{S}^l,$ where $1\leq n \leq m \leq l$ and $X$ is TNHZ in $X_G.$ If $F$ is connected and three cup products from $u_1u_2,u_1u_3,u_1u_4,u_1u_5,u_2u_3$ and $u_2u_4$ of generators of $H^*(F)$ are nonzero. Then, $F$ must be one of the following:
	\begin{enumerate}
		
		
		\item $F \sim_2 {\mathbb{Z}_2[x,y]}/{<x^6,y^3,x^4+y^2>},$ where deg $x=q_1$ and deg $y\in \{ q_2,q_3\}.$
					\item $F \sim_2 {\mathbb{Z}_2[x,y,z]}/{<x^{2+i},y^{2+j},z^{2+k},a_i(yz+x^i),a_j(xz+y^{j}),a_k(xy +z^k)>},$ where $i,j,k\in \{0,2\}$ or $i=j=k=1,a_0=0,a_1=a_2=1$ and deg $x=q_1,$ deg $y=q_2~\&$ deg $z\in \{q_3, q_4\}.$ In particular, for $i=j=k=0,$ $F \sim_2 \mathbb{S}^{r_1} \times \mathbb{S}^{r_2} \times \mathbb{S}^{r_3},$ where $r_1\leq n, r_2\leq m$ and $r_3\leq l.$ 
		 
	\end{enumerate}
\end{theorem}
\begin{proof}
	Here, we have ${6\choose 3}=20$ cases depending on which three nonzero cup products are taken from $u_1u_2,u_1u_3,u_1u_4,u_1u_5,u_2u_3$ and $u_2u_4.$ Out of these twenty cases, the following four cases are possible: (1) $u_1u_2\neq0, u_1u_3\neq0~\&~ u_1u_4\neq0,$ (2) $u_1u_2\neq0, u_1u_3\neq0~\&~u_1u_5\neq0,$ (3) $u_1u_2\neq0, u_1u_3\neq0~\&~u_2u_3\neq0,$ and  (4) $u_1u_2\neq0, u_1u_4\neq0~\&~u_2u_4\neq0.$\\
	Case (1)  $u_1u_2 \neq 0, u_1u_3\neq 0$ and $u_1u_4\neq 0.$\\
	In this case, we must have $u_1u_3=u_4,u_1u_2=u_5,$ and $u_1u_4=u_6.$ Thus, $u_1u_4=u_1^2u_3$ and $u_1u_2^2=u_1^3u_3=u_1u_3^2$ is the generator of $H^r(F).$ So, $u_3^2,u_2^2 $ and $u_1^3$ are nonzero. Since, $u_1u_5=0,$ we must get $u_3=u_1^2.$ Thus, $u_4=u_1^3,$ $u_6=u_1^4,$ and $u_1^5=u_2^2u_1.$ Hence, the cohomology ring $H^*(F)$ is isomorphic to  
		$${\mathbb{Z}_2[u_1,u_2]}/{<u_1^6,u_2^3,u_2^2+u_1^4>,}$$ where deg $u_1=q_1$ and deg $u_2=q_2.$ This realize possibility (1) for $x=u_1,y=u_2$ and deg $y=q_2.$ \\
		Similarly, for Case (2) $u_1u_2 \neq 0, u_1u_3\neq 0,$ $u_1u_5\neq 0,$ we realizes possibility (1) for $x=u_1,y=u_3$ and deg $y=q_3.$\\
	Case(3) $u_1u_2 \neq 0, u_1u_3\neq 0$ and $u_2u_3\neq 0.$\\
	As the cup products $u_1u_4,u_1u_5$ and $u_2u_3$ are zero, only the  following subcases are possible: (i) $u_1u_2=u_4,u_1u_3=u_5~\&~u_2u_3=u_6,$ (ii) $u_1u_2=u_5,u_1u_3=u_6~\&~u_2u_3=u_4,$ and (iii) $u_1u_2=u_6,u_1u_3=u_4~\&~u_2u_3=u_5.$\\
	 Subcase (i) $u_1u_2=u_4,u_1u_3=u_5$ and $u_2u_3=u_6.$ Then, $u_1u_2u_3$ is the generator of $H^r(F).$ Now, if $u_1^2,u_2^2$ and $u_3^2$ are all zero, then $F\sim_2 \mathbb{S}^{q_1} \times \mathbb{S}^{q_2} \times \mathbb{S}^{q_3}.$ By Proposition \ref{prop1}, we get $r_1\leq n, r_2\leq m$ and $r_3\leq l.$ This realizes possibility (2) for $i=j=k=0.$ 
	  
	  Suppose that  $u_1^2,u_2^2$ and $u_3^2$  are all nonzero. Then, we must get $u_4=u_3^2,u_5=u_2^2$ and $u_6=u_3^2.$ Clearly, $u_1^3=u_2^3=u_3^3=u_1u_2u_3$ and  $u_1^4=u_2^4=u_3^4=0.$ Thus, the cohomology ring $H^*(F)$ is isomorphic to 
	$${\mathbb{Z}_2[u_1,u_2,u_3]}/{<u_1^4,u_2^4,u_3^4,u_1u_2+u_3^2,u_1u_3+u_2^2,u_2u_3+u_1^2 >}, $$ where deg $u_1=q_1,$  deg $u_2=q_2$ and  deg $u_3=q_3.$ This realizes possibility (2) for $i=j=k=2.$
	
	 Suppose that any two of  $u_1^2,u_2^2$ and $u_3^2$ are nonzero. Assume  $u_1^2$ and $u_2^2$ are nonzero and $u_3^2=0.$ Then, $u_5=u_2^2$ and $u_6=u_1^2.$ So, the cohomology ring $H^*(F)$ is isomorphic to
		$${\mathbb{Z}_2[u_1,u_2,u_3]}/{<u_1^4,u_2^4,u_3^2,u_1u_3+u_2^2,u_2u_3+u_1^2 >}, $$ where deg $u_1=q_1,$  deg $u_2=q_2$ and  deg $u_3=q_3.$ This realizes possibility (2) for $i=j=2 ~\&~k=0.$
	 
	 Finally, suppose that  one  of  $u_1^2,u_2^2$ and $u_3^2$ is nonzero. Assume  $u_1^2\neq 0 $ and $u_2^2=u_3^2=0.$ Then, $u_6=u_1^2.$ So, the cohomology ring $H^*(F)$ is isomorphic to
	 $${\mathbb{Z}_2[u_1,u_2,u_3]}/{<u_1^4,u_2^2,u_3^2,u_2u_3+u_1^2 >}, $$ where deg $u_1=q_1,$  deg $u_2=q_2$ and  deg $u_3=q_3.$ This realizes possibility (2) for $i=2~\&~j=k=0.$\\
	 Subcase (ii) $u_2u_3=u_4,u_1u_2=u_5$ and $u_1u_3=u_6.$ Then,   $u_1^2u_3=u_2^2u_1=u_3^2u_2$ generates  $H^r(F).$ Clearly, $u_1^2=u_4,u_3^2=u_5$ and $u_2^2=u_6.$ Thus,
	 the cohomology ring $H^*(F)$ is isomorphic to
	 $${\mathbb{Z}_2[u_1,u_2,u_3]}/{<u_1^3,u_2^3,u_3^3,u_2u_3+u_1^2,u_1u_2+u_3^2,u_1u_3+u_2^2 >}, $$ where deg $u_1=q_1,$  deg $u_2=q_2$ and  deg $u_3=q_3.$ This realizes possibility (2) for $i=j=k=1.$\\
	 Subcase (iii) $u_2u_3=u_5,u_1u_2=u_6$ and $u_1u_3=u_4.$ In this subcase, we get same cohomology ring $H^*(F)$ as in subcase (ii).	\\
	 Case (4) $u_1u_2\neq0, u_1u_4\neq0,u_2u_4\neq0.$\\
	 In this case, we must have $u_1u_2=u_3,u_1u_4=u_5$ and $u_2u_4=u_6.$ So, $u_1u_2u_4$ is generator of $H^r(F).$ If $u_4^2\neq 0,$ then either $u_4^2=u_5$ or $u_6,$ it fails the uniqueness of poincar\'{e} dual of a generating element. So,  $u_4^2=0.$ Now, if $u_1^2=u_2^2=0,$ then   $F\sim_2 \mathbb{S}^{q_1} \times \mathbb{S}^{q_2} \times \mathbb{S}^{q_3}.$ This  realizes possibility (2) for $i=j=k=0.$ If both $u_1^2$ and $u_2^2$ are nonzero, then we get $u_6=u_1^2$ and $u_5=u_2^2.$ Thus,  the cohomology ring $H^*(F)$ is isomorphic to
	 $${\mathbb{Z}_2[u_1,u_2,u_4]}/{<u_1^4,u_2^4,u_4^2,u_1u_4+u_2^2,u_2u_4+u_1^2 >}, $$ where deg $u_1=q_1,$  deg $u_2=q_2$ and  deg $u_4=q_4.$ This  realizes possibility (2) for $i=j=2~\&~k=0.$\\
	 Finally, if any one of $u_1^2$ and $u_2^2$ is nonzero, say $u_1^2\neq 0,$ then  $u_6=u_1^2.$ Thus, the cohomology ring $H^*(F)$ is isomorphic to
	 $${\mathbb{Z}_2[u_1,u_2,u_4]}/{<u_1^4,u_2^2,u_4^2,u_2u_4+u_1^2 >}, $$ where deg $u_1=q_1,$  deg $u_2=q_2$ and  deg $u_4=q_4.$ This realizes possibility (2) for $i=2~\&~j=k=0.$\\
	  It is easy to observe that the remaining sixteen cases can be discarded  by the zero cup products or by the uniqueness of poincar\'{e} duals.	 \end{proof}  
 \begin{theorem}\label{3.5}
 		Let $G=\mathbb{Z}_2$ act on $X\sim_2 \mathbb{S}^n \times \mathbb{S}^m \times \mathbb{S}^l,$ where $1\leq n \leq m \leq l$ and $X$ is TNHZ in $X_G.$ If $F$ is connected and  any four cup products from $u_1u_2,u_1u_3,u_1u_4,u_1u_5,u_2u_3$ and $u_2u_4$  of generators of $H^*(F)$ are nonzero. Then, $F$ must be one of the  following:
 	\begin{enumerate}
 		
 		\item $F \sim_2 \mathbb{P}^3(r_1) \times \mathbb{S}^{r_2},$ where $ r_1 =1,2,4,8.$
 		\item $F \sim_2 {\mathbb{Z}_2[x,y]}/{<x^4,y^4,x^3+y^2,y^2x >},$ where deg $x=q_1$ and deg $y=q_2.$ 
 	\end{enumerate}	
 	\end{theorem}
 \begin{proof}
 	We consider $6 \choose 4$$= 15$ cases, depending on which pair of cup products $u_1u_2,u_1u_3,\\u_1u_4,u_1u_5,u_2u_3,u_2u_4$ are zero and rest four cup products are nonzero. Out of these fifteen cases, the following six cases are possible: (i) $u_1u_3=u_2u_3=0,$ (ii) $u_1u_3=u_1u_5=0,$ (iii) $u_1u_4=u_1u_5=0,$ (iv) $u_1u_4=u_2u_4=0,$
(v) $u_1u_5=u_2u_3=0,$ and (vi) $u_1u_5=u_2u_4=0.$ \\ 
Case (i) $u_1u_3=u_2u_3=0.$ \\ We have $u_1u_2,u_1u_4,u_1u_5$ and $u_2u_4$ are nonzero cup products. Clearly, $u_1u_5=u_6.$ By the uniqueness of poincar\'{e} dual of  $u_1^2,$ we must have $u_1u_4=u_5.$ Similarly,  we get $u_1u_2=u_3$ and $u_1u_5=u_2u_4=u_6.$ Then, $u_1^3u_4=u_1u_2u_4$ is the generator of $H^r(F).$ It is easy to see that, $u_2=u_1^2$ and $u_4^2=0.$ Thus, $0=u_1u_3=u_1^4,$ and hence $F \sim_2 \mathbb{P}^3(q_1) \times \mathbb{S}^{q_4}.$ Note that $q_1 =1,2,4,8.$ This realize possibility (1)\\Similarly, for Case (ii)~\&~ (iii), we get $F \sim_2 \mathbb{P}^3(q_2) \times \mathbb{S}^{q_1},$ for Case (iv), we get $F \sim_2 \mathbb{P}^3(q_1) \times \mathbb{S}^{q_3},$ and for Case (v), we get $F \sim_2 \mathbb{P}^3(q_1) \times \mathbb{S}^{q_2}.$ This also realizes possibility (1).\\
Case (vi) $u_1u_5=u_2u_4=0.$\\
As $u_1u_5=u_2u_4=0,$ we get $u_1u_4=u_6,$ $u_1u_2=u_4,u_1u_3=u_5$ and $u_1u_4=u_2u_3=u_6.$ Thus, $u_1^3u_2=u_1u_2u_3$ generates $H^r(F).$ Clearly, $u_1^2=u_3.$ So, if $u_2^2=0,$ then $F \sim_2 \mathbb{P}^3(q_1) \times \mathbb{S}^{q_2}.$ This realizes posibility (1). If $u_2^2\neq 0,$ then we must have $u_5=u_1^3=u_2^2$ and $u_1^3u_2=u_2^3$ generates $H^r(F).$ We have
$0=u_1u_5=u_1^4$ and  $0=u_2u_4=u_2^2u_1.$ Hence,  the cohomology ring $H^*(F)$ is isomorphic to 
$$ {\mathbb{Z}_2[u_1,u_2]}/{<u_1^4,u_2^4,u_1^3+u_2^2,u_2^2u_1 >},$$ where deg $u_1=q_1$ and deg $u_2=q_2.$ This realizes possibility (2).\\ It is easy to observe that the remaining nine cases are not possible.
\end{proof}
\begin{theorem}\label{3.6}
Let $G=\mathbb{Z}_2$ act on $X\sim_2 \mathbb{S}^n \times \mathbb{S}^m \times \mathbb{S}^l,$ where $1\leq n \leq m \leq l$ and $X$ is TNHZ in $X_G.$ If $F$ is connected and  any five cup products from $u_1u_2,u_1u_3,u_1u_4,u_1u_5,u_2u_3$ and $u_2u_4$  of generators of $H^*(F)$ are nonzero. Then,  $F \sim_2 \mathbb{P}^3(r_1) \times \mathbb{S}^{r_1},$ where $ r_1 =1,2,4,8.$
\end{theorem}
\begin{proof}
	Here, we have six cases depending on  which cup product out of $u_1u_2,u_1u_3,u_1u_4,\\u_1u_5,u_2u_3$ and $u_2u_4$ is zero, and rest five cup products are nonzero. Clearly, the only possible case is $u_1u_5=0.$ We have either $u_2u_4=u_6$ or $u_2u_4=u_5.$\\  If $u_2u_4=u_6,$ then we must have $u_1u_3=u_6,u_1u_4=u_2u_3=u_5$ and $u_1u_2=u_3.$ Thus, $u_1^3u_2=u_2^3u_1$ generates $H^r(F).$ Clearly, $u_1^2=u_2^2=u_4.$ So, we get  $u_1^4=u_2^4=0.$ By the change of basis $d=u_1+u_2,$ we get $\{d, u_2\}$ generates $H^{q_1}(F),$ $\{u_2d, u_2^2\}$ generates $H^{q_3}(F),$ $\{u_2^2d, u_2^3\}$ generates $H^{q_5}(F),$ and $u_2^3d$ generates $H^r(F).$ Hence, $F \sim_2 \mathbb{P}^3(q_1) \times \mathbb{S}^{q_1}.$\\ Now, if $u_2u_4=u_5,$ then we must have $u_1u_3=u_5,u_1u_4=u_2u_3=u_6,$ and  $u_1u_2=u_4.$ So, $u_1^3u_2=u_2^3u_1$ generates $H^r(F).$ Simiarly, by the change of basis, we get  $F \sim_2 \mathbb{P}^3(q_1) \times \mathbb{S}^{q_1}.$ Hence, our claim.
 \end{proof}
If all cup products $u_1u_2,u_1u_3,u_1u_4,u_1u_5,u_2u_3$ and $u_2u_4$ are nonzero, then we get
$u_1u_2=u_3,u_1u_3=u_4,u_1u_4=u_2u_3=u_5,u_1u_5=u_2u_4=u_6.$ So, we must have $u_1^2=u_2,$ and hence $u_1^7=0.$ Thus, we have:
\begin{theorem}
	Let $G=\mathbb{Z}_2$ act on $X\sim_2 \mathbb{S}^n \times \mathbb{S}^m \times \mathbb{S}^l,$ where $1\leq n \leq m \leq l$ and $X$ is TNHZ in $X_G.$ If $F$ is connected and  all cup products $u_1u_2,u_1u_3,u_1u_4,u_1u_5,u_2u_3$ and $u_2u_4$  of generators of $H^*(F)$ are nonzero, then $F \sim_2 \mathbb{P}^7(q),~~~q=1,2,4,8.$ 
\end{theorem}

Let $G=\mathbb{Z}_2$ act on $X\sim_2 \mathbb{S}^{n_1}\times\mathbb{S}^{n_2}\times\cdots \mathbb{S}^{n_k},$ where $1\leq n_1 \leq n_2 \leq\cdots \leq n_k$  and $X$ is TNHZ in $X_G.$ Then, rk$H^*(F)=2^k.$ Suppose that $F$ is connected. Let $0<q_1\leq q_2\leq \cdots \leq q_{2^k-2}<r$ be non vanising dimensions of $H^*(F)$ and  $u_i$ is generator of $q_i^{th}$ cohomology group, and $v$ is generator of $H^r(F).$ WLOG, we may assume that $u_1,\cdots ,u_{2^{k-1}-1}$ has poincar\'{e} duals $u_{2^k-2},\cdots ,u_{2^{k-1}},$ respectively. Thus, $v=u_1u_{2^k-2}=\cdots =u_{2^{k-1}-1}u_{2^{k-1}}.$  Clearly, $u_iu_j=0$ if $i+j\geq 2^k.$ It is easy to derive the following results:
 \begin{theorem}
	Let $G=\mathbb{Z}_2$ act on $X\sim_2 \mathbb{S}^{n_1}\times\mathbb{S}^{n_2}\times\cdots \times \mathbb{S}^{n_k},$ where $1\leq n_1 \leq n_2 \leq\cdots \leq n_k$ and $X$ is TNHZ in $X_G.$ If $F$  is connected and all the cup products  of generators of $H^*(F)$ other than generator of $r^{th}$ cohomology group are zero, then $F \sim_2\underset{k}{\#} \mathbb{P}^k(q), q\leq n_1$ and $q=1,2,4,8.$
\end{theorem}
	\begin{theorem}
		Let $G=\mathbb{Z}_2$ act on $X\sim_2 \mathbb{S}^{n_1}\times\mathbb{S}^{n_2}\times\cdots\times \mathbb{S}^{n_k},$ where $1\leq n_1 \leq n_2 \leq\cdots \leq n_k$ and $X$ is TNHZ in $X_G.$ If $F$  is connected and all the cup products $u_iu_j,$ where $i+j\leq  2^k-2,i\neq j,$ of generators of $H^*(F)$ other than $r^{th}$ cohomology group are nonzero, then $F \sim_2 \mathbb{P}^{2^{k}-1}(q),~~~q=1,2,4,8.$ 		
	\end{theorem}
Next, suppose that $X$ is not TNHZ in $X_G$ and $F$ is nonempty connected fixed point set of involutions on $X.$ So, we have   $0<$ rk $H^*(F)< 8.$ By the Floyd's formula \cite{bredon}, we get rk $H^*(F)=2,4$ or $6.$ If $\pi_1(B_G)$ acts nontrivially on $H^*(X),$ then we have proved the following theorem:
\begin{theorem}\label{3.10}
	Let $G=\mathbb{Z}_2$ act on  $X\sim_2 \mathbb{S}^{n}\times \mathbb{S}^{m}\times  \mathbb{S}^{l}, 1\leq n\leq m \leq l$ and    $X$ is not TNHZ in $X_G.$ If $F$ is nonempty connected fixed point set and  $\pi_1(B_G)$ acts nontrivially on $H^*(X),$ then $F$ must be the one of the following:
	\begin{enumerate}
		\item $F \sim_2 \mathbb{S}^{r_1},0<r_1\leq n+m+l.$
		\item $F \sim_2 \mathbb{S}^{r_1}\times  \mathbb{S}^{r_2},0<r_1\leq min\{2n,l\} ~or~n, 0<q_2\leq~max\{2n,l\}~or ~2m.$
		\item $F \sim_2 \mathbb{P}^3(r_1),r_1\leq ~min\{2n,l\}~or ~n,r_1=1,2,4,8.$ 
		\item $F \sim_2 \mathbb{P}^2(r_1)\# \mathbb{P}^2(r_1),r_1\leq ~min\{2n,l\}~or ~n,r_1=1,2,4,8.$
		\item $F \sim_2 \mathbb{P}^2(r_1)\vee \mathbb{S}^{r_2}, r_1\leq ~min\{2n,l\}~or ~n,0<r_2 \leq 2n+l ~or~n+2m,  r_1=1,2,4,8.$
	\end{enumerate}
\end{theorem}
\begin{proof}
	As  $\pi_1(B_G)=G$ acts nontrivially on $H^*(X),$ we must have one of the following conditions: $(i)~ n=m<l$ $(ii) ~n<m=l$ $(iii)~ n=m=l.$ \\ If $n=m<l,$ then $G$ acts nontrivially on $H^n(X).$ Let  $g$ be  generator of $G.$ We have exactly three possibilities of  nontrivial actions: (1) $g^*(a)=b~\&$ $g^*(b)=a,$ (2) $g^*(a)=a~\&$ $g^*(b)=a+b$ and (3) $g^*(a)=a+b~\&$ $g^*(b)=b.$ First, we consider nontrivial action defined by $g^*(a)=b~\&~g^*(b)=a.$  By the natuality of cup product, we get $G$ also acts nontrivially on $H^{n+l}(X).$ Note that $\sigma=\tau=1+g^*.$ By Proposition \ref{nontrivial}, we get $E_2^{0,i}\cong \mathbb{Z}_2$ and  $E_2^{k,i}= 0~~~\forall~~k>0, ~~ i=n,n+l.$ For  $i\neq n,n+l$, $\pi_1(B_G)$ acts trivially on $H^i(X),$ and hence, $E_2^{k,i} \cong H^k(B_G) \otimes H^i(X)\cong \mathbb{Z}_2 ~~~\forall~~k\geq 0, ~~ i=0,l,2n,2n+l.$   By Proposition \ref{prop 6}, $ 1 \otimes ab$ is a permanent cocycle.
	
	 First, suppose that $E_2^{*,*}\neq E_{\infty}^{*,*}.$ As $F\neq \emptyset,$ we must have $d_{l+1}(1\otimes c)=0.$ If  $d_{l-2n+1}$ is nontrivial, then $d_{2n+l+1}$ must be trivial. Thus rk $H^k(X_G)=2$ for large value of $k,$  and by Proposition \ref{thm 2}, we get  rk $H^*(F)=2.$ So, $F\sim_2 \mathbb{S}^q,$ and by Proposition \ref{prop1}, we get $0<q\leq n+m+l.$ Similarly, if  $d_{l-2n+1}$ is trivial, then  $d_{2n+l+1}$ must be  nontrivial. Clearly, $F\sim_2 \mathbb{S}^q, 0<q\leq n+m+l.$ This realizes possibility (1).
	 
	 Next, supppose that  $E_2^{*,*}= E_{\infty}^{*,*}.$ Then rk $H^*(F)=4.$ Assume that $x\in F,$ and    $\alpha\in H^*(X_G,x_G)$ and $\beta \in H^*(X_G,x_G)$ represent $1 \otimes ab \in E_2^{0,2n}$ and    $1 \otimes c \in E_2^{0,l},$ respectively. Clearly, $\{t^i\alpha, t^i\beta , t^i\alpha \beta\}$ forms a basis for $H^k(X_G,x_G)$ over $\mathbb{Z}_2$-module. Let $0<q_1\leq q_2 \leq q_3$ be the non vanishing dimensions of $H^*(F)$ and $u_i$ generates ${q_i}^{th}$ cohomology group of $F,1\leq i \leq 3.$ Then we write 
	 $$j^*(\alpha)=A_1 t^{2n-q_1}\otimes u_1+A_2 t^{2n-q_2}\otimes u_2+A_3 t^{2n-q_3}\otimes u_3,$$
	 
	 $$ j^*(\beta)=B_1 t^{l-q_1}\otimes u_1+B_2 t^{l-q_2}\otimes u_2+B_3 t^{l-q_3}\otimes u_3,$$ So, we get $$j^*(\alpha\beta)=A_1B_1 t^{2n-2q_1}\otimes u_1^2+(A_1B_2+A_2B_1) t^{2n+l-q_1-q_2}\otimes u_1u_2+A_2B_2 t^{2n+l-2q_2}\otimes u_2^2,$$
	  where $A_i ~\&~B_i$ are in $ \mathbb{Z}_2.$ Clearly, $u_1u_3=u_2u_3=u_3^2=0.$ As $\alpha\beta \neq 0,$   $j^*(\alpha\beta)$ is one of the following: (1) $t^{2n+l-2q_1}\otimes u_1^2,$ (2) $t^{2n+l-2q_2}\otimes u_2^2,$ (3) $t^{2n+l-q_1-q_2}\otimes u_1u_2,$ (4) $t^{2n+l-2q_1}\otimes u_1^2 +t^{2n+l-q_1-q_2}\otimes u_1u_2,$ (5) $t^{2n+l-2q_2}\otimes u_2^2 +t^{2n+l-q_1-q_2}\otimes u_1u_2,$ and (6) $t^{2n+l-2q_1}\otimes u_1^2+t^{2n+l-2q_2}\otimes u_2^2.$ \\ It is easy to observe that the cases (5) and (6) are not possible. If $j^*(\alpha\beta)=t^{2n+l-2q_1}\otimes u_1^2+ t^{2n+l-2q_2}\otimes u_2^2,$ then either rk $H^*(F)>4$ or $j^*(\alpha\beta)=0,$ a contradiction. And, if $j^*(\alpha\beta)=t^{2n+l-2q_2}\otimes u_2^2+ t^{2n+l-q_1-q_2}\otimes u_1u_2,$ then $j^*(\alpha \beta)=0,$ again a contradiction.\\ Now, suppose that $j^*(\alpha\beta)=t^{2n+l-2q_1}\otimes u_1^2+ t^{2n+l-q_1-q_2}\otimes u_1u_2.$ Then, if $u_1u_2=u_1^2,$ then $j^*(\alpha \beta )=0,$ a contradiction. So, $u_2=u_1^2,$ and hence $F\sim_2 \mathbb{P}^3(q_1),q_1<2n$ or $q_1< l,$ according as $2n <l$ or $l<2n,$ respectively.  Note that, if $q_1=2n$ or $q_1=l,$ then $u_1=a|F,$ and  we get $u_1^2=0,$ a contradiction.   This realizes possibility (3).\\ Next, suppose that $j^*(\alpha\beta)= t^{2n+l-q_1-q_2}\otimes u_1u_2.$ In this case, we have four conditions  $(i)~u_1^2=0~\&~u_2^2 =0,~(ii)~u_1^2\neq 0~\&~u_2^2 =0, ~(iii)~u_1^2= 0~\&~u_2^2 \neq 0, $ and $~(iv)~u_1^2\neq 0~\&~u_2^2 \neq0. $ Clearly, if $u_1^2=0~\&~u_2^2 =0,$ then $F \sim_2 \mathbb{S}^{q_1}\times  \mathbb{S}^{q_2}.$ Now, if $u_1^2\neq 0~\&~u_2^2 =0,$ then for $u_2=u_1^2,$ we have $F \sim_2 \mathbb{P}^3(q),$ and for $u_1u_2=u_1^2,$ we get $q_1=q_2$ and by the change of basis $d=u_1+u_2,$ we have $F \sim_2 \mathbb{P}^2(q)\# \mathbb{P}^2(q).$ Next, if $u_1^2= 0~\&~u_2^2 \neq 0,$ then $u_1u_2=u_2^2,$ and we get $F \sim_2 \mathbb{P}^2(q)\# \mathbb{P}^2(q).$ Finally, if $u_1^2\neq 0~\&~u_2^2 \neq 0,$ then we must have $u_1u_2=u_2^2=u_1^2,$ and  we get   $F \sim_2 \mathbb{P}^2(q)\vee \mathbb{S}^q.$ \\
	  Now, suppose that $j^*(\alpha\beta)=t^{2n+l-2q_1}\otimes u_1^2,$ then as  above we have four conditions:  $(i)~u_2^2=0~\&~u_1u_2 =0,~(ii)~u_2^2\neq 0~\&~u_1u_2 =0, ~(iii)~u_2^2= 0~\&~u_1u_2 \neq 0, $ and $~(iv)~u_2^2\neq 0~\&~u_1u_2 \neq0.$ If $u_2^2=0~\&~u_1u_2 =0,$ then $F \sim_2 \mathbb{P}^2(q_1)\vee \mathbb{S}^{q_2}$ or $F \sim_2 \mathbb{P}^2(q_1)\vee \mathbb{S}^{q_3},$ accordingly $u_1^2=u_3$ or  $u_1^2=u_2,$ respectively. If $u_2^2\neq 0~\&~u_1u_2 =0,$ then we must have $u_1^2=u_2^2,$ and hence $F \sim_2 \mathbb{P}^2(q)\# \mathbb{P}^2(q).$ If $u_2^2= 0~\&~u_1u_2 \neq 0,$ then for $u_2=u_1^2$ we get $F \sim_2 \mathbb{P}^3(q),$ and for $u_1u_2=u_1^2,$ we get $q_1=q_2$ and by the change of basis $d=u_1+u_2,$ we have $F \sim_2 \mathbb{P}^2(q)\# \mathbb{P}^2(q).$  If $u_2^2\neq 0~\&~u_1u_2 \neq0,$ then we must have $u_1u_2=u_2^2=u_1^2,$ and  we get   $F \sim_2 \mathbb{P}^2(q)\vee \mathbb{S}^q.$ Clearly, all the inequalities hold by Propostion \ref{prop1}.\\
	  Finally, suppose that  $j^*(\alpha\beta)=t^{2n+l-2q_1}\otimes u_1^2.$ This case is similar to the case when $j^*(\alpha\beta)=t^{2n+l-2q_1}\otimes u_1^2.$ \\
  For  the other two nontrivial actions of $G$ on $H^n(X),$ we get  the same  result.\\
Next, if $n<m=l$ or $n=m=l,$ then we get the same  possibilities as in the case when $n=m<l.$	\end{proof}	
\begin{remark}
If $\pi_1(B_G)$ acts trivially on $H^*(X)$ and rk $H^*(F)=2$ or $4,$ then we get same possibilities of nonempty connected fixed point sets as in Theorem  \ref{3.10}. Note that for the smooth manifolds in Theorem \ref{3.10}, we get  possibilities (1)-(4) only.
	
\end{remark}
 
\subsection{Possibilities of disconnected fixed point sets}

Let $G=\mathbb{Z}_2$ act on  $X\sim_2 \mathbb{S}^{n}\times \mathbb{S}^{m}\times  \mathbb{S}^{l}, 1\leq n\leq m \leq l.$  In this section, we determine the possibilities of a disconnected fixed point set $F$ when $X$ is TNHZ in $X_G.$ We have $\sum$ rk $H^i(F) =8.$ Clearly, $F$ has at most eight components. By Proposition \ref{thm 1}, each component $F_0$ of $F$ is a poincar\'{e} duality space of formal dimension, say $r\leq n+m+l.$ If rk $H^*(F_0)=1$ or $2$ or $3,$ then  $F_0$ have the mod $2$ cohomology of a point or a sphere or a projective space of height two, respectively. Now, if rk $H^*(F_0)=4,$ then  $F_0$ have the mod $2$ cohomology of either a product of spheres $\mathbb{S}^{q_1} \times  \mathbb{S}^{q_2},$ or a projective space of height three $ \mathbb{P}^3(q)$ or a connected sum of two projective spaces of height two $ \mathbb{P}^2(q) \# \mathbb{P}^2(q)$ \cite{Bredon}. For any component $F_0$ of $F$ with rk $H^*(F_0)\leq 4,$   we have following theorem:

\begin{theorem}\label{3.12}
	Let $G=\mathbb{Z}_2$ act on $X\sim_2 \mathbb{S}^n \times \mathbb{S}^m \times \mathbb{S}^l,$ where $1\leq n \leq m \leq l$ and $X$ is TNHZ in $X_G.$ If $F$ is disconnected  and the rank of cohomology ring of a component of $F$ is at most four, then $F$ must be one of the following:
	\begin{enumerate}
\item $F \sim_2 \underset{i=0}{\overset{3}{\sqcup}}F_i,$ where $F_i \sim_2\mathbb{S}^{r_i},0\leq i \leq 3.$

\item $F \sim_2 \underset{i=0}{\overset{2}{\sqcup}}F_i,$ where $F_0 \sim_2 \mathbb{P}^2(r_1) $ or $\mathbb{S}^{r_2}+pt,$ $F_1 \sim_2 \mathbb{P}^2(r_3)$ and $F_2\sim_2\mathbb{S}^{r_4}.$
\item $F \sim_2 \underset{i=0}{\overset{2}{\sqcup}} F_i,$ where $F_0 \sim_2 \mathbb{S}^{r_1} \times  \mathbb{S}^{r_2}$ or $\mathbb{P}^3(r_1) $ or $\mathbb{P}^2(r_1) \# \mathbb{P}^2(r_2)$ and  $F_i \sim_2 \mathbb{S}^{r_i},i=1,2.$
\item $F \sim_2 \underset{i=0}{\overset{2}{\sqcup}} F_i,$ where $F_0 \sim_2 \mathbb{S}^{r_1} \times  \mathbb{S}^{r_2}$ or $\mathbb{P}^3(r_1) $ or $\mathbb{P}^2(r_1) \# \mathbb{P}^2(r_1),$  $F_1 \sim_2 \mathbb{P}^2(r_2)$  and $F_2 \sim_2 pt.$
\item $F \sim_2 \underset{i=0}{\overset{1}{\sqcup}} F_i,$ where $F_i \sim_2 \mathbb{S}^{r_1} \times  \mathbb{S}^{r_2}$ or $\mathbb{P}^3(r_1) $ or $\mathbb{P}^2(r_1) \# \mathbb{P}^2(r_1),i=0,1.$
	\end{enumerate} 
\end{theorem}

\noindent Further, if rk $H^*(F_0)>4,$ where $F_0$ is a component of $F,$ then we have following theorems:
\begin{theorem}\label{3.13}
	Let $G=\mathbb{Z}_2$ act on $X\sim_2 \mathbb{S}^n \times \mathbb{S}^m \times \mathbb{S}^l,$ where $1\leq n \leq m \leq l$ and $X$ is TNHZ in $X_G.$ If $F$ is disconnected  and  the rank of cohomology ring of a component of $F$ is five, then $F$ must be one of the following:
	\begin{enumerate} \item $F\sim_2\underset{i=0}{\overset{1}{\sqcup}} F_i,$ where $F_0 \sim_2 \underset{3}{\#} \mathbb{P}^2(r_1)$ or $\mathbb{P}^2(r_1) \# ( {\mathbb{S}^{r_2} \times \mathbb{S}^{r_3} })$ or $\mathbb{P}^4(r_1)$ and $F_1 \sim_2 \mathbb{P}^2(r_1).$
	\item	$F\sim_2\underset{i=0}{\overset{2}{\sqcup}} F_i,$ where $F_0\sim_2 \underset{3}{\#} \mathbb{P}^2(r_1)$ or $\mathbb{P}^2(r_1) \# ( {\mathbb{S}^{r_2} \times \mathbb{S}^{r_3} })$ or $\mathbb{P}^4(r_1)$ and $F_1 \sim_2 \mathbb{S}^{r_4}$  or $F_2\sim_2 pt .$
		\end{enumerate}
\end{theorem}
\begin{proof}
	Let $F_0$ be the component of $F$ such that rk $H^*(F_0)=5.$ Suppose  that $0<q_1\leq q_2\leq q_3 < r$ are non vanishing dimensions of $H^*(F_0),$  and $u_i$ generates $H^{q_i}(F_0),1\leq i\leq 3$ and $v$ generates $H^{r}(F_0).$ Clearly, $u_iv=0,1\leq i\leq 3.$ \\ First, we consider that each generator $u_i, 1\leq i \leq 3,$ is poincar\'{e} dual of itself. In this case, we get $u_1^2=u_2^2=u_3^2$  generates  $H^{r}(F),$ and hence $F_0\sim_2 \underset{3}{\#} \mathbb{P}^2(q_1).$\\ Now, we consider that exactly  one generator is poincar\'{e} dual of itself.  Let $u_2$ be poincar\'{e} dual of itself. So, we have  $u_3$ is  the  poincar\'{e} dual of $u_1.$ Then, $u_1u_3=u_2^2$ generates $H^{r}(F_0).$ Clearly, $u_2u_3=u_3^2=0.$ Suppose that $u_1u_2 = 0.$ Further, if $u_1^2 \neq 0,$ then we must have  $u_3=u_1^2,$ and  hence $F_0 \sim_2 \mathbb{P}^3(q_1)\# \mathbb{P}^2(q_2),$ which is not possible. 
	If $u_1^2=0,$ then  $F_0\sim_2 \mathbb{P}^2(q_2) \# ( {\mathbb{S}^{q_1} \times \mathbb{S}^{q_3} }).$ Now, suppose $u_1u_2 \neq 0.$ Then,  we get $u_3=u_1u_2.$ So,  $u_1^2 u_2=u_2^2$ generates $H^{r}(F_0)$ and  we must have $u_2=u_1^2.$ Hence, $F_0 \sim_2 \mathbb{P}^4(q_1),$ where $q_1=1,2,4,8.$ For the remaining cases, when $u_1$ or $u_3$ is poincar\'{e} dual of itself, we only get  $F_0\sim_2 \mathbb{P}^2(q) \# ( {\mathbb{S}^{q} \times \mathbb{S}^{q} }).$\\
	 Since rk $H^{*}(F)=8$ and rk $H^*(F_0)=5,$ then $F$ have at most four components. Clearly, the components other than $F_0$  have the mod $2$ cohomology of a point or a sphere or a projective space of height two. Hence, our claim. 
\end{proof}
\begin{theorem}\label{3.14}
	Let $G=\mathbb{Z}_2$ act on $X\sim_2 \mathbb{S}^n \times \mathbb{S}^m \times \mathbb{S}^l,$ where $1\leq n \leq m \leq l$ and $X$ is TNHZ in $X_G.$ If $F$ is disconnected  and one of the rank of cohomology ring of a component of $F$ is six, then $F\sim_2\underset{i=1}{\overset{2}{\sqcup}} F_i,$ where $F_1 \sim_2 \underset{2}{\#} \mathbb{P}^3(r_1)$ or $\mathbb{P}^3(r_1) \# (\mathbb{S}^{r_2}\times \mathbb{S}^{r_3})$ or $\mathbb{P}^2(r_1)\times \mathbb{S}^{r_2}$ or $\mathbb{P}^5(r_1)$ and $F_2\sim_2 \mathbb{S}^{r_2}.$  \end{theorem}
\begin{proof}
	Let $F_0$ be the component of $F$ with rk $H^*(F_0)=6.$  Assume that $0<q_1\leq q_2\leq q_3 \leq q_4 < r$ are  non vanishing dimensions of $H^*(F_0),$ and $u_i$  generates $H^{q_i}(F_0),1\leq i\leq 4$ and $v$ generates $H^{r}(F_0).$  If each generator $u_i$ is poincar\'{e} dual of itself or any two generators are poincar\'{e} dual of itself, then this contradicts Proposition \ref{puppe}. As $q_i\leq q_{i+1},1\leq i \leq 3,$   we must have  $u_4$ and $u_3$ are  poincar\'{e} duals of $u_1$ and $u_2,$ respectively. So, $u_1u_4=u_2u_3$ generates $H^{r}(F_0).$ Clearly, $u_2u_4=u_3u_4=0.$ So, the cup products $u_1u_2$ and $ u_1u_3$ are either  zero or nonzero.
	
First,	Suppose that   $u_1u_2= u_1u_3=0.$ Clearly,  $u_3^2=u_4^2=0.$ Now, we have $u_1^2$ and $u_2^2$ are either zero or nonzero. If $u_1^2=u_2^2=0,$ then $F\sim_2 (\mathbb{S}^{q_2}\times \mathbb{S}^{q_3})\#(\mathbb{S}^{q_1}\times \mathbb{S}^{q_4}),$ which contradicts Proposition \ref{puppe}. If $u_1^2\neq 0$ and $u_2^2\neq 0,$ then  we  get $u_4=u_1^2$ and $u_3=u_2^2.$ Thus, $F\sim_2 \mathbb{P}^3(q_1)\#\mathbb{P}^3(q_2).$ If  $u_1^2\neq 0$ and  $u_2^2=0,$  then we get $F\sim_2 \mathbb{P}^3(q_1)\#(\mathbb{S}^{q_2}\times \mathbb{S}^{q_3}).$ Similarly, if $u_1^2=0$ and  $u_2^2\neq0,$ then $F\sim_2 \mathbb{P}^3(q_1)\#(\mathbb{S}^{q_1}\times \mathbb{S}^{q_4}).$  

Now, suppose that  $u_1u_2 \neq 0$ and $u_1u_3 =0.$ Then, we have either $u_1u_2=u_3$ or $u_1u_2=u_4.$ If $u_1u_2=u_3,$ then  $u_1u_4=u_1u_2^2$
	 generates   $H^{r}(F_0).$ So, $u_2^2 \neq 0,$ and we get  $u_4=u_2^2.$ Further, if $u_1^2=0,$ then we get $F\sim_2 \mathbb{P}^2(q_1)\times \mathbb{S}^{q_2}.$ For $u_1^2\neq 0,$ we get  $u_1^2=u_2^2.$ By the change of basis $d=u_1+u_2,$ we get $F\sim_2 \mathbb{P}^2(q_1)\times \mathbb{S}^{q_1}.$ Similarly, if  $u_1u_2=u_4,$ then $F\sim_2 \mathbb{P}^2(q_1)\times \mathbb{S}^{q_2}.$
	 
	  Now,  suppose that $u_1u_2 = 0$ and $u_1u_3 \neq 0.$ Then,  $u_1u_3=u_4$ and $u_1^2u_3$  generates   $H^{r}(F_0).$ Clearly, $u_3^2=0$ and $u_2=u_1^2.$ Thus, $F\sim_2 \mathbb{P}^2(q_1)\times \mathbb{S}^{q_3}.$
	  
	   Finally, suppose that both $u_1u_2 $ and $u_1u_3 $ are nonzero. In this case, we  get  $u_3=u_1u_2$ and $u_4=u_1u_3.$ Clearly, $u_2=u_1^2,$ and $F\sim_2 \mathbb{P}^5(q_1),$ where $q_1=1,2,4,8.$\\ Note that, the component other than $F_0$ have sphere cohomology. Hence, our claim.
\end{proof}
\begin{theorem}\label{3.15}
	Let $G=\mathbb{Z}_2$ act on $X\sim_2 \mathbb{S}^n \times \mathbb{S}^m \times \mathbb{S}^l,$ where $1\leq n \leq m \leq l$ and $X$ is TNHZ in $X_G.$ If $F$ is disconnected  and one of the rank of cohomology ring of component of $F$ is seven, then $F$ must be one of the following:
		\begin{enumerate} \item $F\sim_2\underset{i=1}{\overset{2}{\sqcup}} F_i,$ where $F_1 \sim_2 (\mathbb{P}^2(r_1)\times \mathbb{S}^{r_2}) \# \mathbb{P}^2(r_3)$ or $\mathbb{P}^4(r_1) \# (\mathbb{S}^{r_2} \times \mathbb{S}^{r_3})$ or $\mathbb{P}^6(r_1)$ and $F_2 \sim_2 pt.$
			\item $  F \sim_2 {\mathbb{Z}_2[u_1,u_2]}/{<u_1^5,u_2^3,u_1^3+u_2^2,u_1^2u_2>}+pt,$ where deg $u_1=r_1$ and 	deg $u_2=r_2.$
	\end{enumerate}
	\end{theorem}
\begin{proof}
	Let $F_0$ be the component of $F$ such that rk $H^*(F_0)=7.$ Clearly, the  component of $F$ other than $F_0$ have  point cohomology. Assume that $0<q_1\leq q_2\leq q_3 \leq q_4 \leq q_5 < r$ are the non vanishing dimensions of $H^*(F_0)$ with  $u_i,1\leq i\leq 5$ is  generator of $H^{q_i}(F_0)$ and $v$ generates $H^{r}(F_0).$ If each generator $u_i$ is poincar\'{e} dual of itself or any three genertors are poincar\'{e} duals of itself,  then we get a  contradiction of  Proposition \ref{puppe}. Clearly,  $u_3$ is poincar\'{e} dual of itself and  $u_5$ and $u_4$ are  poincar\'{e} duals of $u_1$ and $u_2,$ respectively. 
	Therefore, $u_1u_5=u_2u_4=u_3^2$  generates $H^{r}(F_0).$ Note that the cup products $v^2=u_2u_5=u_3u_4=u_3u_5=u_4^2=u_5^2=u_iv=0, 1\leq i\leq 5.$ The remaining cup products $u_1u_2,u_1u_3,u_1u_4,$ and $u_2u_3$ may be zero or nonzero. Now, we considered  cases accordingly to how many cup products are nonzero.\\ It is easy to observe that,  if any three cup products are nonzero or none of the cup products are nonzero then these cases are not possible. So, we consider the following cases:\\
	\textbf{Case(1):} If any one cup product from $u_1u_2,u_1u_3,u_1u_4,u_2u_3$ is nonzero.\\
		First,	Suppose that  $u_1u_2 \neq 0,$ then either $u_1u_2 =u_4$ or $u_1u_2=u_5.$ Let $u_1u_2 =u_4.$ Then, $v= u_1u_5=u_2^2u_1=u_3^2.$  If $u_2^2=u_3,$ then $u_3^2=u_2^4 \neq 0$ but $u_2u_3= u_2^3=0,$ a contradiction. By  the uniqueness of poincar\'{e} dual of $u_2,$ we get  $u_2^2=u_5.$ If $u_1^2 = 0,$ then  we get $F\sim_2 (\mathbb{P}^2(q_1)\times \mathbb{S}^{q_2}) \# \mathbb{P}^2(q_3),$ and $q_1,q_3 \in \{1,2,4,8\}.$ If $u_1^2 \neq  0,$ then we must have $u_2^2=u_1^2,$ and by the change of basis $d=u_1+u_2,$  we get $F \sim_2(\mathbb{P}^2(q_1)\times \mathbb{S}^{q_1}) \# \mathbb{P}^2(q_3).$ Similarly, if $u_1u_2=u_5,$ then $F \sim_2(\mathbb{P}^2(q_1)\times \mathbb{S}^{q_2}) \# \mathbb{P}^2(q_3).$ 
	
Suppose that $u_1u_3\neq 0,$ then $u_1u_3=u_5.$ So, $u_1^2u_3=u_2u_4=u_3^2=v.$ Clearly, $u_1^2=u_3 $ and $u_2^2=0.$ Thus,  $F \sim_2\mathbb{P}^4(q_1)\# (\mathbb{S}^{q_2} \times  \mathbb{S}^{q_2}), q_1=1,2,4,8.$  \\Similarly, if $u_1u_4 \neq 0,$ then we get $F\sim_2 (\mathbb{P}^2(q_1)\times \mathbb{S}^{q_4}) \# \mathbb{P}^2(q_3),$ and if  $u_2u_3 \neq 0,$ then we get $F \sim_2\mathbb{P}^4(q_2)\# (\mathbb{S}^{q_1} \times  \mathbb{S}^{q_5}).$ These realizes possibility (1). \\
\textbf{Case(2):} If any two  cup products from $u_1u_2,u_1u_3,u_1u_4,u_2u_3$ are nonzero.\\
It is easy observe that the only possible nonzero pair is $u_1u_2~\&~u_1u_3.$ We must have  $u_1u_3=u_6$ and $u_1u_2=u_4$. Thus, $u_1^2u_3=u_2^2u_1=u_3^2=v.$ Clearly, $u_3=u_1^2$ and $u_1u_3=u_2^2.$ Thus, the cohomology ring of $H^*(F_0)$ isomorphic to $${\mathbb{Z}_2[u_1,u_2]}/{<u_1^5,u_2^3,u_1^3+u_2^2,u_1^2u_2>},$$ where deg $u_1=q_1$ and deg $u_2=q_2.$ This realizes possibility (2).\\
\textbf{Case(3):} If all cup products $u_1u_2,u_1u_3,u_1u_4,u_2u_3$ are nonzero.\\
	In this case, we must have $u_1u_2=u_3,$ $u_1u_3=u_4,$ $u_1u_4=u_2u_3=u_5$ and $u_1^2=u_2.$ Hence, $F_0\sim_2 \mathbb{P}^6(q),$ where $q=1,2,4,8.$ This realizes possibility (1).
\end{proof}
Next, suppose that $X$ is not TNHZ in $X_G$ and $F$ is disconnected fixed point set of involutions on $X.$ A proof of the following theorem is similar to proof of Theorem \ref{3.10}. 
\begin{theorem}\label{thm3.15}
	Let $G=\mathbb{Z}_2$ act on  $X\sim_2 \mathbb{S}^{n}\times \mathbb{S}^{m}\times  \mathbb{S}^{l}, 1\leq n\leq m \leq l$ and    $X$ is not TNHZ in $X_G.$ If $F$ is disconnected fixed point set and  $\pi_1(B_G)$ acts nontrivially on $H^*(X),$ then $F$ must be the one of the following:
	\begin{enumerate}
		\item $F \sim_2 \mathbb{S}^0.$
		\item $F \sim_2 \mathbb{S}^{r_1}+  \mathbb{S}^{r_2}, 0\leq r_1,r_2\leq n+m+l.$
			\item $F \sim_2 \mathbb{P}^2(r_1)+ pt,  r_1=1,2,4,8, r_1\leq n+m+l.$
		\item $F \sim_2 \mathbb{S}^{r_1}\vee \mathbb{S}^{r_2} + pt, 0\leq r_1,r_2\leq n+m+l.$ 
	
	\end{enumerate}
\end{theorem}
Note that if $\pi_1(B_G)$ acts trivially on $H^*(X)$ and rk $H^*(F)=2$ or $4,$ then we get same possibilities of disconnected fixed point sets as in Theorem  \ref{thm3.15}. For the smooth manifolds possibilities  (3) and (4) will not occur.

Next, we give some examples to realizes above Theorems:
\begin{example} Let  $G=\mathbb{Z}_2$ act on $Y\sim_2 \mathbb{S}^n\times \mathbb{S}^m,n\leq m.$ If $Y$ is TNHZ in $Y_G,$ then the possibilities of mod $2$ cohomology of fixed point sets  are $\mathbb{S}^{r_1}\times \mathbb{S}^{r_2}, \mathbb{P}^2(r_1) \# \mathbb{P}^2(r_1), \mathbb{P}^3(r_1),\\~ \mathbb{P}^2(r_1)+pt$ and $ \mathbb{S}^{r_1}+\mathbb{S}^{r_2},$ and if $Y$ is not TNHZ in $Y_G,$ then   $F\sim_2\mathbb{S}^{r_1}$ \cite{Bredon}. Recall that if $G$ acts on $Z\sim_2 \mathbb{S}^l,l>0,$ then  mod $2$ cohomology of the fixed point set is $\mathbb{S}^{r_3},r_3\geq 0.$ After taking the diagonal action of $G$ on  $X=Y \times Z\sim_2 \mathbb{S}^n\times \mathbb{S}^m \times \mathbb{S}^l,$ we  get the possibilities of mod $2$ cohomology of fixed point sets are   $\mathbb{S}^{r_1}\times \mathbb{S}^{r_2}\times \mathbb{S}^{r_3}, (\mathbb{P}^2(r_1) \# \mathbb{P}^2(r_1))\times \mathbb{S}^{r_3}, \mathbb{P}^3(r_1)\times \mathbb{S}^{r_3}, (\mathbb{P}^2(r_1)+pt)\times \mathbb{S}^{r_3}$ and $ (\mathbb{S}^{r_1}+\mathbb{S}^{r_2})\times \mathbb{S}^{r_3},$ in the case when $X$ is TNHZ in $X_G$ and $F\sim_2 \mathbb{S}^{r_1} \times \mathbb{S}^{r_3},$ in the case when  $X$ is not TNHZ in $X_G.$ This realizes possibility (2) of Theorems $\ref{3.3},$ $\ref{3.4}$ and $\ref{3.10},$ possibility (1) of Theorems $\ref{3.5}$ and $\ref{3.6},$ possibilities (1) and (5) of Theorem $\ref{3.12},$ and Theorem $\ref{3.14}.$
\end{example}

\end{document}